\documentclass[12pt]{amsart}
\usepackage[all]{xy}
\usepackage{amssymb}
\usepackage{parskip}
\usepackage{tikz,pgf}
\usepackage{mathtools}
\usepackage{booktabs}
\usetikzlibrary{calc, matrix, positioning, backgrounds}

\usepackage{color}
\def\red{}
\def\blue{}

\newtheorem{theorem}{Theorem}[section]

\newtheorem{lemma}[theorem]{Lemma}
\newtheorem{corollary}[theorem]{Corollary}

\theoremstyle{definition}
\newtheorem{definition}[theorem]{Definition}
\newtheorem{example}[theorem]{Example}
 
\theoremstyle{remark}
\newtheorem{remark}[theorem]{Remark}

\usepackage{fourier}
\usepackage[T1]{fontenc}
\usepackage[utf8]{inputenc}
\title[On L-S category and topological complexity of \MakeUppercase{\protect\mdkn}]{On Lusternik-Schnirelmann category and topological complexity of no $k$-equal manifolds 
}
\thanks{The second author was supported by a Conacyt scholarship.}
\author{Jes\'us Gonz\'alez \ \ and \ \ Jos\'e Luis Le\'on-Medina}

\newcommand{\mybin}[2]{\ensuremath{\scriptsize\Big(\begin{array}{@{}c@{}} #1\\[-4pt] #2\end{array}\Big)}}
\newcommand{\floor}[1]{\left\lfloor #1 \right\rfloor}
\newcommand{\nn}{\ensuremath{\mathbf{n}}}
\newcommand{\mdkn}{\ensuremath{M_d^{(k)}(n)}}

\newcommand{\bigast}{\mathop{\raisebox{-0.6ex}{\scalebox{1.5}{$\ast$}}}}
\newcommand{\mfloor}[1]{\ensuremath{\left\lceil #1-1\right\rceil}}

\def\TC{\protect\operatorname{TC}}
\def\cat{\protect\operatorname{cat}}

\usepackage{url}
\begin{document}
\maketitle

\begin{abstract}
We compute the Lusternik-Schnirelmann category and the topological complexity of no $k$-equal manifolds $\mdkn$ for certain values of $d$, $k$ and $n$. This includes instances where $\mdkn$ is known to be rationally non-formal. The key ingredient in our computations is the knowledge of the cohomology ring $H^*(\mdkn)$ as described by Dobrinskaya and Turchin in \cite{dobri2015}. A fine tuning comes from the use of obstruction theory techniques.
\end{abstract}



\section{Introduction and main results}\label{seccionintroduccion}

The no $k$-equal manifold $\mdkn$ is defined as the complement in $\big(\mathbb{R}^d\big)^n$ of the diagonal-subspace arrangement formed by the union of subspaces
\[
A_I = \big\lbrace(x_1,\dots,x_n)\in \big(\mathbb{R}^d\big)^n \,|\, x_{i_1} = \cdots = x_{i_k}\big\rbrace,
\]
where $I=\{i_1,\dots, i_k\}$ runs through all cardinality-$k$ subsets of the segment $\nn=\{1,2,\dots, n\}$. 
For $k=2$ this construction yields the classical and extensively studied configuration space of $n$ distinct ordered points in $\mathbb{R}^d$. In this paper we will only deal with the case $k\geq3$. Further, as $\mdkn=\big(\mathbb{R}^d\big)^n $ for $n<k$ and $M_d^{(k)}(k)\simeq\mathbb{S}^{dk-d-1}$, we will restrict our attention to the combinatorially more interesting case $k<n$. The aim of this paper is to extend the work in \cite{Gonzalez2019} where the Lusternik-Schnirelmann category ($\cat$), topological complexity ($\TC$) and sequential topological complexity ($\TC_s$) of $\mdkn$ is computed for $d=1$. Here we address the more subtle problem for $d\geq2$. We prove:

\begin{theorem}\label{omnibus}
For $d\geq2$, $n>k\geq3$ and $n-(k-1)\floor{\frac{n}{k}}\leq\frac{dk-2}{d-1}$, $$\TC_s(\mdkn)=s\floor{\frac{n}{k}}$$ for any $s\geq1$. Here we use $\TC_1$ and $\TC_2$ as substitutes for $\cat$ and $\TC$ respectively.
\end{theorem}

Note that the number $n-(k-1)\floor{\frac{n}{k}}$ in Theorem~\ref{omnibus} is the sum of the quotient and the remainder in the division of $n$ by $k$. Theorem~\ref{omnibus} is most difficult to prove when the latter sum agrees in fact with $\frac{dk-2}{d-1}$, for then the argument is based on techniques of obstruction theory.

A central piece of information in our proof arguments comes from Do\-brins\-kaya-Turchin's description in \cite[Section 6]{dobri2015} of the cohomology ring of $\mdkn$ in terms of certain {admissible} $k$-forests. We review their description in Section~\ref{secciondescripciondecohomologia} below.

It is known from~\cite{Miller2012} that not all manifolds $\mdkn$ are rationally formal when $d=2$. In fact, non-trivial Massey products holding in these spaces would seem to play a key role in the complete determination of their $\TC_s$-invariants. In this direction it is interesting to remark that the hypothesis
\begin{equation}\label{nuestracondicion}
n-(k-1)\floor{\frac{n}{k}}\leq 2k-2
\end{equation}
in Theorem~\ref{omnibus} (for $d=2$) is less restrictive than the inequality
\begin{equation}\label{condiciondeMiller}
 n+\floor{\frac{n}{k}}(k-2)<6k-9
\end{equation}
coming from~\cite{Miller2012} as a sufficient condition for the rational formality of $M^{(k)}_2(n)$. For instance, with $k=4$,~(\ref{condiciondeMiller}) amounts to $n\leq10$, while~(\ref{nuestracondicion}) holds for $n\leq 24$, except for $n\in\{19,22,23\}$. Thus, Theorem~\ref{omnibus} describes cat and TC invariants for manifolds that (as far as it is currently known) might fail to be rationally formal. For instance, while the work in \cite{Miller2012} shows that $M_2^{(3)}(n)$ is rationally formal if and only if $n\leq 6$, our calculations show that the equality TC$_s\left(M_2^{(3)}(n)\right) = s\floor{\frac{n}{k}}$ holds for $n\leq 12$ with the possible exception of $n=11$. 

\section{The cohomology ring $H^\ast(\mdkn)$}\label{secciondescripciondecohomologia}

All cohomology groups in the following sections are taken with coefficients in the ring $R$, where either $R=\mathbb{Z}$ or $R=\mathbb{Z}_2$. Assertions made without specifying the ring $R$ are meant to hold for both options of $R$. Of course, the several sign specifications below can be ignored when $R=\mathbb{Z}_2$.

\begin{definition}\label{defkforest} A $k$-forest on $\nn$ (or simply a $k$-forest) is an acyclic graph with two types of vertices, square and round, each containing a certain subset of $\nn$.
\begin{itemize}
\item A square vertex must contain $k-1$ elements of $\nn$, and cannot be an isolated vertex; in fact the set of immediate neighbors of a square vertex must contain a round vertex.
\item A round vertex must contain a single element of $\nn$, and must be either an isolated vertex or have valency 1, in which case it must be connected to a square vertex. 
\end{itemize}
We require that the subsets of integers inside the various vertices of a $k$-forest form a disjoint partition of $\nn$.

An orientation for a $k$-forest consists of three ingredients:
\begin{itemize}
    \item[(a)] An orientation for each edge;
    \item[(b)] An ordering for elements inside each square vertex;
    \item[(c)] An ordering for the \emph{orientation set}, i.\,e., the set consisting of all edges and all square vertices.
\end{itemize}

Square vertices are declared to have degree $d(k-2)$, while edges are declared to have degree $d-1$. The degree of a $k$-forest is then defined as the sum of the degrees of its square vertices and edges. 
\end{definition} 

\begin{example}
Consider the following $5$-forest of dimension $28$ (with $d=3$ and $n=14$):
    
\begin{center}
\begin{tikzpicture}
\node (A) [rectangle, draw, minimum width=10mm] at (-1.15,0) {$1\enskip 3\enskip 9\enskip 8$};
\node (B) [rectangle, draw, minimum width=10mm] at (1.15,0) {$4\enskip 5\enskip 6\enskip 13$};
\node (r1) [circle, draw, inner sep=0.5pt,minimum width=5mm] at (-2.75,0) {$2$};
\node (r2) [circle, draw, inner sep=0.5pt,minimum width=5mm] at (-1.5,1.25) {$11$};
\node (r3) [circle, draw, inner sep=0.5pt,minimum width=5mm] at (-0.5,1.25) {$12$};
\node (r4) [circle, draw, inner sep=0.5pt,minimum width=5mm] at (0.75,1.25) {$7$};
\node (r5) [circle, draw, inner sep=0.5pt,minimum width=5mm] at (1.75,1.25) {$10$};
\node (r6) [circle, draw, inner sep=0.5pt,minimum width=5mm] at (2.75,0) {$14$};
\draw[line width=0.5pt,-latex] ([xshift=-2mm]A.north)--(r2.south)node[midway,left]{\tiny $1$};
\node at ([xshift=-1mm,yshift=-1mm]A.south west) {\tiny $2$};
\draw[line width=0.5pt,latex-] ([xshift=2mm]A.north)--(r3.south)node[midway,right]{\tiny $3$};
\draw[line width=0.5pt,latex-] ([xshift=-2mm]B.north)--(r4.south)node[midway,left]{\tiny $4$};
\draw[line width=0.5pt,-latex] (A.east)--(B.west)node[midway,above]{\tiny $5$};
\draw[line width=0.5pt,-latex] ([xshift=2mm]B.north)--(r5.south)node[midway,left]{\tiny $6$};
\node at ([xshift=1mm,yshift=-1mm]B.south east) {\tiny $7$};
\end{tikzpicture}
\end{center}
The tiny numbers attached to square vertices and edges indicate the ordering in the orientation set.
\end{example}

We agree that, in a picture like this, the ordering of elements inside a square vertex is spelled out by listing the elements from left to right. As indicated in the following result, the $5$-forest above represents a $28$-dimen\-sional cohomology class in $M_3^{(5)}(14)$.

\begin{theorem}[{\cite[Theorem 6.1]{dobri2015}}]\label{sum} Let $d\geq 2$, $k\geq 3$ and $n\geq 1$. Additively, the cohomology of $\mdkn$ is the graded torsion-free $R$-module spanned by the oriented $k$-forests on $\nn$ subject to the relations listed below.
\begin{enumerate}
\item Orientation relations:
    \begin{itemize}
        \item[(i)] Permuting the order of the orientation set introduces the Koszul sign induced by the permutation (with respect to the degrees of the elements of the orientation set).
        \item[(ii)] A permutation $\sigma \in \Sigma_{k-1}$ of the elements inside a square vertex introduces the sign $\epsilon(\sigma)^{d}$, where $\epsilon(\sigma)$ stands for the sign of $\sigma$.
        \item[(iii)] Reversing the orientation of an edge introduces the sign $(-1)^d$.
    \end{itemize}
\item Three-term relation:
    \begin{center}
    \begin{tikzpicture}
    \node (A) [rectangle, draw, minimum width=8mm] at (-1,0) {$A$};
    \node (B) [rectangle, draw, minimum width=8mm] at (0,1) {$B$};
    \node (C) [rectangle, draw, minimum width=8mm] at (1,0) {$C$};
    \draw[line width=0.5pt,-latex] (A.north)--(B.south west)node[midway, xshift=-2mm, yshift=1mm]{\tiny $1$}; 
    \draw[line width=0.5pt,-latex] (B.south east)--(C.north)node[midway,xshift=2mm,yshift=1mm]{\tiny $2$};
    \begin{scope}[xshift=3.25cm]
    \node (A2) [rectangle, draw, minimum width=8mm] at (-1,0) {$A$};
    \node (B2) [rectangle, draw, minimum width=8mm] at (0,1) {$B$};
    \node (C2) [rectangle, draw, minimum width=8mm] at (1,0) {$C$};
    \draw[line width=0.5pt,latex-] (A2.east)--(C2.west)node[midway,below]{\tiny $2$}; 
    \draw[line width=0.5pt,-latex] (B2.south east)--(C2.north)node[midway, xshift=1mm, yshift=1.25mm] {\tiny $1$};
    \end{scope}
    \begin{scope}[xshift=6.5cm]
    \node (A3) [rectangle, draw, minimum width=8mm] at (-1,0) {$A$};
    \node (B3) [rectangle, draw, minimum width=8mm] at (0,1) {$B$};
    \node (C3) [rectangle, draw, minimum width=8mm] at (1,0) {$C$};
    \draw[line width=0.5pt,-latex] (A3.north)--(B3.south west)node[midway,xshift=-1.25mm, yshift=1mm]{\tiny $2$};
    \draw[line width=0.5pt,latex-] (A3.east)--(C3.west)node[midway,below]{\tiny $1$};
    \end{scope}
    \coordinate (D) at ($(C.east)!0.5!(A2.west)$);
    \coordinate (E) at ($(C2.east)!0.5!(A3.west)$);
    \node at ([yshift=6mm]D) {$+$};
    \node at ([yshift=6mm]E) {$+$};
    \node at ([shift={(6mm,6mm)}]C3.east) {$=$\enskip $0$.};
    \end{tikzpicture}
    \end{center}
    These pictures are local in the sense that we have three oriented $k$-forests that are identical except for the disposition of their edges connecting vertices $A$, $B$ and $C$, and for the relative ordering of these edges (indicated by the numbers shown) within the corresponding orientation sets.
\item Dual generalized Jacobi relation:
    \begin{center}
    \begin{tikzpicture}
    \node at (-.42,-1.5) {\red{$\cdots$}}; \node at (1.77,-1.5) {\red{$\cdots$}};
    \node (rec) [rectangle, draw] at (0,0) {$i_1$ $i_2$ $\cdots$ $i_{k-2}$ $j_\ell$};
    \node (r1) [circle, draw, inner sep=0.5pt,minimum width=7.5mm] at (-2,-1.5) {$j_{\scriptscriptstyle 1}$};
    \node (r2) [circle, draw, inner sep=0.5pt,minimum width=7.5mm] at (-1.1,-1.5) {$j_{\scriptscriptstyle 2}$};
    \node (r3) [circle, draw, inner sep=0.5pt,minimum width=7.5mm] at (.2,-1.5) {$j_{\scriptscriptstyle\footnotesize\ell-1}$};
    \node (r4) [circle, draw, inner sep=0.5pt,minimum width=7.5mm] at (1.1,-1.5) {$j_{\scriptscriptstyle\ell+1}$};
    \node (r5) [circle, draw, inner sep=0.5pt,minimum width=7.5mm] at (2.4,-1.5) {$j_{\scriptscriptstyle m}$};
    \draw[line width=0.5pt,-latex] (rec.south west)--(r1.north)node[midway,xshift=-2mm,yshift=0mm]{\tiny $1$};
    \draw[line width=0.5pt,-latex] ($(rec.south west)!0.5!(rec.south)$)--(r2.north)node[midway,xshift=-2mm,yshift=0mm]{\tiny $2$};
    \draw[line width=0.5pt,-latex] (rec.south)--(r3.north)node[midway,xshift=-5mm,yshift=-0.5mm]{\tiny $\cdots$};
    \draw[line width=0.5pt,-latex] ($(rec.south east)!0.5!(rec.south)$)--(r4.north);
    \draw[line width=0.5pt,-latex] (rec.south east)--(r5.north)node[midway,xshift=-5mm,yshift=-0.5mm]{\tiny $\cdots$}node[midway,xshift=5.2mm,yshift=0mm]{\tiny $m-1$};
    \node at (-3.5,-0.75) {$0=\displaystyle \sum_{\ell=1}^m (-1)^{\ell(d-1)}$};
    \node at (3.35,-0.75) {.};
    \end{tikzpicture}
    \end{center}
    Again pictures are local. Moreover, in the global picture, the square vertex cannot be connected to other (non shown) round vertices.
\end{enumerate}
Furthermore, basic oriented $k$-forests form a graded basis for the cohomology groups $H^\ast(\mdkn)$.
\end{theorem}

In the final assertion of Theorem~\ref{sum}, an oriented $k$-forest is called \emph{basic} if, ignoring orientations of edges, any of its connected components is either an isolated round vertex or, else, a ``semilinear" tree, i.e.~a tree having the (global) form
\begin{center}
\begin{tikzpicture} 
\node (r1) [circle, draw, minimum width=5mm] at (-0.75,1) {};
\node (r2) [circle, draw, minimum width=5mm] at (0.75,1) {};
\node (A1) [rectangle, draw, minimum width=1cm] at (0,0) {$A_1$};
\draw[line width=0.5pt] (A1.north west)--(r1.285);
\node at (0,1) {$\cdots$};
\draw[line width=0.5pt] (A1.north east)--(r2.255);
\begin{scope}[xshift=2.5cm]
\node (r3) [circle, draw, minimum width=5mm] at (-0.75,1) {};
\node (r4) [circle, draw, minimum width=5mm] at (0.75,1) {};
\node (A2) [rectangle, draw, minimum width=1cm] at (0,0) {$A_2$};
\draw[line width=0.5pt] (A2.north west)--(r3.285);
\node at (0,1) {$\cdots$};
\draw[line width=0.5pt] (A2.north east)--(r4.255);
\end{scope}
\draw[line width=0.5pt] (A1.east)--(A2.west);
\begin{scope}[xshift=6cm]
\node (r5) [circle, draw, minimum width=5mm] at (-0.75,1) {};
\node (r6) [circle, draw, minimum width=5mm] at (0.75,1) {};
\node (A3) [rectangle, draw, minimum width=1cm] at (0,0) {$A_s$};
\draw[line width=0.5pt] (A3.north west)--(r5.285);
\node at (0,1) {$\cdots$};
\draw[line width=0.5pt] (A3.north east)--(r6.255);
\end{scope}
\draw[line width=0.5pt] (A2.east)--(A3.west)node[midway, fill=white]{$\dots$};
\end{tikzpicture}
\end{center}
with the following additional requirements:
\begin{itemize}
\item $A_1<A_2<\cdots <A_s$ in the orientation set.
\item For a portion of the form
\begin{center}
\begin{tikzpicture} 
\node (r1) [circle, draw, minimum width=5mm] at (-0.75,1) {};
\node (r2) [circle, draw, minimum width=5mm] at (0.75,1) {};
\node (A1) [rectangle, draw, minimum width=1cm] at (0,0) {$A_i$};
\draw[line width=0.5pt] (A1.north west)--(r1.285);
\node at (0,1) {$\cdots$};
\draw[line width=0.5pt] (A1.north east)--(r2.255);
\end{tikzpicture}
\end{center}
the elements inside the square vertex appear in their natural order. Likewise, the ordering (in the orientation set) of the edges attaching round vertices to the square vertex agrees with the na\-tural order of the integers inside those round vertices.
\item The largest of the integers inside the round vertices attached to each square vertex $A_i$ is larger than any of the integers inside $A_i$.
\item The minimal element in the semilinear tree appears either inside $A_1$ or inside a round vertex attached to $A_1$.
\end{itemize}

\begin{remark}
The above description arises as a result of considering the Borel-Moore homology of $\mdkn$, where each $k$-forest represents a locally finite chain in $\mdkn$ whose boundary lies outside $\mdkn$, and sums correspond to unions of these chains. Signs described in Theorem~\ref{sum} arise from a consistent management of chain orientations. For example, the tree-term relation
\begin{center}
    \begin{tikzpicture}
    \node (A) [rectangle, draw, minimum width=8mm] at (-1,0) {$A$};
    \node (B) [rectangle, draw, minimum width=8mm] at (0,1) {$B$};
    \node (C) [rectangle, draw, minimum width=8mm] at (1,0) {$C$};
    \draw[line width=0.5pt,-latex] (A.north)--(B.south west)node[midway, xshift=-2mm, yshift=1mm]{\tiny $1$}; 
    \draw[line width=0.5pt,-latex] (B.south east)--(C.north)node[midway,xshift=2mm,yshift=1mm]{\tiny $2$};
    \begin{scope}[xshift=3.25cm]
    \node (A2) [rectangle, draw, minimum width=8mm] at (-1,0) {$A$};
    \node (B2) [rectangle, draw, minimum width=8mm] at (0,1) {$B$};
    \node (C2) [rectangle, draw, minimum width=8mm] at (1,0) {$C$};
    \draw[line width=0.5pt,latex-] (A2.east)--(C2.west)node[midway,below]{\tiny $2$}; 
    \draw[line width=0.5pt,-latex] (B2.south east)--(C2.north)node[midway, xshift=1mm, yshift=1.25mm] {\tiny $1$};
    \end{scope}
    \begin{scope}[xshift=6.5cm]
    \node (A3) [rectangle, draw, minimum width=8mm] at (-1,0) {$A$};
    \node (B3) [rectangle, draw, minimum width=8mm] at (0,1) {$B$};
    \node (C3) [rectangle, draw, minimum width=8mm] at (1,0) {$C$};
    \draw[line width=0.5pt,-latex] (A3.north)--(B3.south west)node[midway,xshift=-1.25mm, yshift=1mm]{\tiny $2$};
    \draw[line width=0.5pt,latex-] (A3.east)--(C3.west)node[midway,below]{\tiny $1$};
    \end{scope}
    \coordinate (D) at ($(C.east)!0.5!(A2.west)$);
    \coordinate (E) at ($(C2.east)!0.5!(A3.west)$);
    \node at ([yshift=6mm]D) {$+$};
    \node at ([yshift=6mm]E) {$+$};
    \node at ([shift={(6mm,6mm)}]C3.east) {$=$\enskip $0$.};
    \end{tikzpicture}
    \end{center}
is a rearrangement, under the sign conventions, of the element that corresponds to the union of two locally finite chains:
\begin{center}
    \begin{tikzpicture}
    \node (A) [rectangle, draw, minimum width=8mm] at (-1,0) {$A$};
    \node (B) [rectangle, draw, minimum width=8mm] at (0,1) {$B$};
    \node (C) [rectangle, draw, minimum width=8mm] at (1,0) {$C$};
    \draw[line width=0.5pt,latex-] (A.north)--(B.south west)node[midway, xshift=-2mm, yshift=1mm]{\tiny $1$}; 
    \draw[line width=0.5pt,-latex] (B.south east)--(C.north)node[midway,xshift=2mm,yshift=1mm]{\tiny $2$};
    \begin{scope}[xshift=3.75cm]
    \node (A2) [rectangle, draw, minimum width=8mm] at (-1,0) {$A$};
    \node (B2) [rectangle, draw, minimum width=8mm] at (0,1) {$B$};
    \node (C2) [rectangle, draw, minimum width=8mm] at (1,0) {$C$};
    \draw[line width=0.5pt,latex-] (A2.east)--(C2.west)node[midway,below]{\tiny $1$}; 
    \draw[line width=0.5pt,-latex] (B2.south east)--(C2.north)node[midway, xshift=1mm, yshift=1.25mm] {\tiny $2$};
    \end{scope}
    \begin{scope}[xshift=7.3cm]
    \node (A3) [rectangle, draw, minimum width=8mm] at (-1,0) {$A$};
    \node (B3) [rectangle, draw, minimum width=8mm] at (0,1) {$B$};
    \node (C3) [rectangle, draw, minimum width=8mm] at (1,0) {$C$};
    \draw[line width=0.5pt,latex-] (A3.north)--(B3.south west)node[midway,xshift=-1.25mm, yshift=1mm]{\tiny $1$};
    \draw[line width=0.5pt,-latex] (A3.east)--(C3.west)node[midway,below]{\tiny $2$};
    \end{scope}
    \coordinate (D) at ($(C.east)!0.5!(A2.west)$);
    \coordinate (E) at ($(C2.east)!0.5!(A3.west)$);
    \node at ([yshift=6mm]D) {$=$};
    \node at ([yshift=6mm]E) {$+$};
    \end{tikzpicture}\enskip\raisebox{0.4em}{.}
    \end{center}
Similarly, the generalized Jacobi relation is a boundary relation in terms of $k$-forests. Under such a (Poincaré duality) approach, cohomology cup-products are readable as intersection products in Borel-Moore homology. The product structure is spelled out in Theorem~\ref{prod}. For further details see~\cite{dobri2015}.
\end{remark}

\begin{remark}
Observe that interchanging a square vertex with an edge does not modify the sign of the forest because $(d-1)\cdot d(k-2)$ is always even. Thus, the orientation set only keeps track of the relative orientation of square vertices and edges.
\end{remark}

The multiplicative structure in the cohomology of $\mdkn$ is dictated by the following result:

\begin{theorem}[{\cite[Theorem 7.1]{dobri2015}}]\label{prod} For $n,k,d$ as in Theorem~\ref{sum}, let $T_1$, $T_2 \in H^\ast(\mdkn)$ be two oriented $k$-forests. The cup product of $T_1$ and $T_2$ is zero if either of the following three conditions holds:
\begin{itemize}
\item[(1)] There exist a square vertex $A$ in $T_1$ and a square vertex $B$ in $T_2$ such that $A\cap B \neq \emptyset$.

In case that no square vertex of $T_1$ intersects a square vertex of $T_2$, we define the superposition $T_1\cup T_2$ as the graph obtained by superposition of the vertices of $T_1$ and $T_2$ with the convention that if some integer $i \in \nn$ lies in a round vertex in, say, $T_{2}$ as well as in a square vertex $A$ in $T_{1}$, then $i$ appears in $T_1\cup T_2$ inside the corresponding square vertex $A$, and if there were some oriented edge in $T_{2}$ from the round vertex containing $i$ to some square vertex $B$, then a corresponding oriented edge between vertices $A$ and $B$ in $T_1\cup T_2$ would have to be added:
\begin{center}
\begin{tikzpicture}[scale=0.85]
\node (A) [circle, draw, inner sep=0.5pt,minimum width=5mm] {$i$};
\node (B) [rectangle, draw, minimum width=1.6cm, minimum height=0.7cm] at (-0.4,0) {\hspace*{-6mm}$A$};
\node (B) [rectangle, draw, minimum width=1.4cm, minimum height=0.7cm] at (3,0) {$B$};
\draw[line width=0.5pt,-latex] (A.east)--(B.west);
\end{tikzpicture}\enskip\raisebox{0em}{.}
\end{center}
(This of course might lead to multiple edges between two given square vertices in $T_1\cup T_2$.)

\item[(2)] $T_1 \cup T_2$ has unoriented cycles (for instance if two square vertices of $T_1 \cup T_2$ are joined by multiple edges). 
\item[(3)] $T_1\cup T_2$ has a square vertex with no round vertex attached.
\end{itemize}
Otherwise, $T_1 \cdot T_2 = T_1\cup T_2$, the superposition of the $k$-forests with orientation set given by the concatenation of the orientation sets of the factors, and with the convention that, if $T_1\cup T_2$ is not a $k$-forest (in the sense of Definition~\ref{defkforest}), so that $T_1 \cup T_2$ has one or several round vertices of valency $2$ (with two square vertices as its immediate neighboring vertices), then we use repeatedly the following form of the three-term relation to write $T_1\cup T_2$ as a sum of $k$-forests:
\begin{equation}\label{productrelation}
\tag{R}
\begin{tikzpicture}[baseline=(current bounding box.center)]
\node (A) [rectangle, draw, minimum width=8mm] at (-1,0) {$A$};
\node (B) [circle, draw, minimum width=5mm] at (0,1) {};
\node (C) [rectangle, draw, minimum width=8mm] at (1,0) {$B$};
\draw[line width=0.5pt,latex-] (A.north)--(B.210)node[midway,xshift=-2mm,yshift=1mm]{\tiny $1$};
\draw[line width=0.5pt,-latex] (B.330)--(C.north)node[midway,xshift=2mm, yshift=1mm]{\tiny $2$};
\begin{scope}[xshift=3.25cm]
\node (A2) [rectangle, draw, minimum width=8mm] at (-1,0) {$A$};
\node (B2) [circle, draw, minimum width=5mm] at (0,1) {};
\node (C2) [rectangle, draw, minimum width=8mm] at (1,0) {$B$};
\draw[line width=0.5pt,-latex] (A2.east)--(C2.west)node[midway,below]{\tiny $2$};
\draw[line width=0.5pt,-latex] (B2.210)--(A2.north)node[midway,xshift=-2mm,yshift=1mm]{\tiny $1$};
\end{scope}
\begin{scope}[xshift=6.5cm]
\node (A3) [rectangle, draw, minimum width=8mm] at (-1,0) {$A$};
\node (B3) [circle, draw, minimum width=5mm] at (0,1) {};
\node (C3) [rectangle, draw, minimum width=8mm] at (1,0) {$B$};
\draw[line width=0.5pt,latex-] (C3.north)--(B3.330)node[midway,xshift=2mm,yshift=1mm]{\tiny $2$};
\draw[line width=0.5pt,latex-] (A3.east)--(C3.west)node[midway,below]{\tiny $1$};
\end{scope}
\coordinate (D) at ($(C.east)!0.5!(A2.west)$);
\coordinate (E) at ($(C2.east)!0.5!(A3.west)$);
\node at ([yshift=6mm]D) {$=$};
\node at ([yshift=6mm]E) {$+$};
\end{tikzpicture}\enskip\raisebox{-1.85em}{.}
\end{equation}
As above, these pictures are local.
\end{theorem}

\begin{example}\label{ejemplito1}
Item~(3) in Theorem~\ref{prod} might have to be used in the iterative process of applying~(\ref{productrelation}) to write $T_1\cup T_2$ as a sum of (basic) $k$-forests. For instance, if the pictures in~(\ref{productrelation}) are in fact global (omitting isolated round vertices), then the two summands on the right of~(\ref{productrelation}) would vanish in view of item~(3) in Theorem~\ref{prod}.
\end{example}

Relevant for us is to note that the algebra $H^\ast(\mdkn)$ is generated by basic oriented $k$-forests having a single square vertex; such a generator will be said to be \emph{elementary}. Explicitly, a basic oriented $k$-forest is, up to sign, the product of its connected components. In turn, each such connected component is, up to sign, a product of elementary oriented $k$-forests. For example, the basic oriented $3$-forest

\begin{equation}\label{exampleproducbasic}
\tag{F}
\begin{tikzpicture}[baseline=(current bounding box.center)]
\node (A) [rectangle, draw, minimum width=10mm] at (-1,0) {$1\enskip 2$};
\node (B) [rectangle, draw, minimum width=10mm] at (1,0) {$4\enskip 5$};
\node (C) [rectangle, draw, minimum width=10mm] at (3,0) {$7\enskip 8$};
\node (r1) [circle, draw, inner sep=0.5pt,minimum width=5mm] at (-1,1) {$3$};
\node (r2) [circle, draw, inner sep=0.5pt,minimum width=5mm] at (1,1) {$6$};
\node (r3) [circle, draw, inner sep=0.5pt,minimum width=5mm] at (3,1) {$9$};
\node at ([xshift=1mm,yshift=-1mm]A.south east) {\tiny $1$};
\node at ([xshift=1mm,yshift=-1mm]B.south east) {\tiny $2$};
\node at ([xshift=1mm,yshift=-1mm]C.south east) {\tiny $3$};
\draw[line width=0.5pt,-latex] (A.north)--(r1.south)node[midway,left]{\tiny $4$};
\draw[line width=0.5pt,-latex] (B.north)--(r2.south)node[midway,left]{\tiny $6$};
\draw[line width=0.5pt,-latex] (C.north)--(r3.south)node[midway,left]{\tiny $8$};
\draw[line width=0.5pt,-latex] (A.east)--(B.west)node[midway,above]{\tiny $5$};
\draw[line width=0.5pt,-latex] (B.east)--(C.west)node[midway,above]{\tiny $7$};
\end{tikzpicture}
\end{equation}
is the product
\[
\left(\rule[-6mm]{0pt}{12mm}\right.\  
\begin{tikzpicture}[baseline=(current bounding box.center)]
\node (A) [rectangle, draw, minimum width=10mm] at (-1,0) {$1\enskip 2$};
\node (r1) [circle, draw, inner sep=0.5pt,minimum width=5mm] at (-1.35,1) {$3$};
\node (r2) [circle, draw, inner sep=0.5pt,minimum width=5mm] at (-0.65,1) {$4$};
\node at ([xshift=1mm,yshift=-1mm]A.south east) {\tiny $1$};
\draw[line width=0.5pt,-latex] ([xshift=-1mm]A.north)--(r1.south)node[midway,left]{\tiny $2$};
\draw[line width=0.5pt,-latex] ([xshift=1mm]A.north)--(r2.south)node[midway,right]{\tiny $3$};
\end{tikzpicture}\left.\rule[-6mm]{0pt}{12mm}\right)
\left(\rule[-6mm]{0pt}{12mm}\right.\  
\begin{tikzpicture}[baseline=(current bounding box.center)]
\node (A) [rectangle, draw, minimum width=10mm] at (-1,0) {$4\enskip 5$};
\node (r1) [circle, draw, inner sep=0.5pt,minimum width=5mm] at (-1.35,1) {$6$};
\node (r2) [circle, draw, inner sep=0.5pt,minimum width=5mm] at (-0.65,1) {$7$};
\node at ([xshift=1mm,yshift=-1mm]A.south east) {\tiny $1$};
\draw[line width=0.5pt,-latex] ([xshift=-1mm]A.north)--(r1.south)node[midway,left]{\tiny $2$};
\draw[line width=0.5pt,-latex] ([xshift=1mm]A.north)--(r2.south)node[midway,right]{\tiny $3$};
\end{tikzpicture}\left.\rule[-6mm]{0pt}{12mm}\right)
\left(\rule[-6mm]{0pt}{12mm}\right.\  
\begin{tikzpicture}[baseline=(current bounding box.center)]
\node (A) [rectangle, draw, minimum width=10mm] at (-1,0) {$7\enskip 8$};
\node (r2) [circle, draw, inner sep=0.5pt,minimum width=5mm] at (-1,1) {$9$};
\node at ([xshift=1mm,yshift=-1mm]A.south east) {\tiny $1$};
\draw[line width=0.5pt,-latex] ([xshift=0mm]A.north)--(r2.south)node[midway,right]{\tiny $2$};
\end{tikzpicture}\left.\rule[-6mm]{0pt}{12mm}\right)
\]
where we have omitted to write isolated round vertices.

In some arguments below we will consider $\mathbb{Z}_2$ representations of $k$-forests so to avoid sign and orientation conventions.  In those cases, a positive-degree connected component of a basic $k$-forest is a semilinear undirected tree
\begin{center}
\begin{tikzpicture} 
\node (r1) [circle, draw, minimum width=5mm] at (-0.75,1) {};
\node (r2) [circle, draw, minimum width=5mm] at (0.75,1) {};
\node (A1) [rectangle, draw, minimum width=1cm] at (0,0) {$A_1$};
\draw[line width=0.5pt] (A1.north west)--(r1.285);
\node at (0,1) {$\cdots$};
\draw[line width=0.5pt] (A1.north east)--(r2.255);
\begin{scope}[xshift=2.5cm]
\node (r3) [circle, draw, minimum width=5mm] at (-0.75,1) {};
\node (r4) [circle, draw, minimum width=5mm] at (0.75,1) {};
\node (A2) [rectangle, draw, minimum width=1cm] at (0,0) {$A_2$};
\draw[line width=0.5pt] (A2.north west)--(r3.285);
\node at (0,1) {$\cdots$};
\draw[line width=0.5pt] (A2.north east)--(r4.255);
\end{scope}
\draw[line width=0.5pt] (A1.east)--(A2.west);
\begin{scope}[xshift=6cm]
\node (r5) [circle, draw, minimum width=5mm] at (-0.75,1) {};
\node (r6) [circle, draw, minimum width=5mm] at (0.75,1) {};
\node (A3) [rectangle, draw, minimum width=1cm] at (0,0) {$A_s$};
\draw[line width=0.5pt] (A3.north west)--(r5.285);
\node at (0,1) {$\cdots$};
\draw[line width=0.5pt] (A3.north east)--(r6.255);
\end{scope}
\draw[line width=0.5pt] (A2.east)--(A3.west)node[midway, fill=white]{$\dots$};
\end{tikzpicture}
\end{center}
where one of the integers in round vertices attached to each $A_i$ is larger than any of the vertices inside $A_i$, and where the smallest of the integers in the vertices of the component lies either in $A_1$ or $A_s$ or in a round vertex attached to $A_1$ or to $A_s$. Further, such a $\mathbb{Z}_2$-equipped component will be lifted canonically to a $\mathbb{Z}$-equiped component. Namely, integers inside a square vertex are taken with their natural order; edges attaching round vertices to a given square vertex are taken with the natural order of the integers they contain; edges attaching a round vertex to a square vertex are oriented to point toward the round vertex; edges between square vertices are oriented to follow a linear path starting from the portion 
\begin{center}
\begin{tikzpicture} 
\node (r1) [circle, draw, minimum width=5mm] at (-0.75,1) {};
\node (r2) [circle, draw, minimum width=5mm] at (0.75,1) {};
\node (A1) [rectangle, draw, minimum width=1cm] at (0,0) {\phantom{$A_1$}};
\draw[line width=0.5pt,-latex] (A1.north west)--(r1.285);
\node at (0,1) {$\cdots$};
\draw[line width=0.5pt,-latex] (A1.north east)--(r2.255);
\end{tikzpicture}
\end{center}
containing the smallest integer of the component; lastly, regarding the global ordering of the orientation set, square vertices are ordered in the direction of the (already) oriented linear tree, while any edge departing from $A_i$ is declared to be smaller than any edge departing from $A_j$ provided $i<j$. In the case $i=j$ we declare that a potential edge 
\begin{center}
\begin{tikzpicture} 
\node (A1) [rectangle, draw, minimum width=1cm] at (0,0) {{$A_i$}};
\node (A2) [rectangle, draw, minimum width=1cm] at (2,0) {{$A_{i+1}$}};
\draw[line width=0.5pt,-latex] (A1.east)--(A2.west);
\end{tikzpicture}
\end{center}
is larger than any other edge departing from $A_i$. By definition, such a $\mathbb{Z}$-equipment of the component yields a basic component. See for instance the orientations and ordering shown in \eqref{exampleproducbasic}.

\begin{example}\label{exampleYuzvinsky} The non-trivial cohomology groups of the manifold $M_2^{(3)}(6)$ are described by Yuzvinsky in \cite[page 1944]{yuz2002} as follows:
\begin{center}
\renewcommand{\arraystretch}{1.25}
\begin{tabular}{*{2}{c}}\toprule
$H^\ast(M_2^{(3)}(6))$ & rank \\\midrule
$H^3(M_2^{(3)}(6))$ & 20 \\
$H^4(M_2^{(3)}(6))$ & 45 \\
$H^5(M_2^{(3)}(6))$ & 36 \\
$H_1^6(M_2^{(3)}(6))$ & 10 \\
$H_2^6(M_2^{(3)}(6))$ &  10\\
$H^7(M_2^{(3)}(6))$ & 10 \\\bottomrule
\end{tabular}
\end{center}
where $H^6(M_2^{(3)}(6))=H_1^6(M_2^{(3)}(6)) \oplus H_2^6(M_2^{(3)}(6))$. Yuzvinsky shows that there are non-zero cup products only in $H^3(M_2^{(3)}(6))\otimes H^3(M_2^{(3)}(6)) \twoheadrightarrow H_2^6(M_2^{(3)}(6))$ and $H^3(M_2^{(3)}(6))\otimes H^4(M_2^{(3)}(6)) \twoheadrightarrow H^7(M_2^{(3)}(6))$. All these facts are rather transparent using the description of $H^\ast(M_2^{(3)}(6))$ in terms of $3$-forests:
\begin{itemize}
\item The smallest dimension where the cohomology is non trivial is given by the minimal dimension of an elementary $3$-forest, a square vertex with one round vertex attached, and it has dimension $3$. As three numbers determine an elementary $3$-forest, there are $\mybin{6}{3}=20$ basis elements in $H^3(M_2^{(3)}(6))$. Here and below, non-explicited orientation and orderings are taken as explained above.
\item The next cohomological dimension is generated by elementary $3$-forests with one square and two round vertices attached. In this case we can first select $4$ numbers to fill in the square and round vertices. From those numbers, the greatest value is forced to be in a round vertex so it only remains to determine the value of the other round vertex. Therefore there are $\mybin{6}{4}\mybin{3}{1}=45$ basis elements in $H^4(M_2^{(3)}(6))$.
\item The rank of $H^5(M_2^{(3)}(6))$ is obtained similarly, in this case there are $\mybin{6}{5}\mybin{4}{2} = 36$ basis elements.
\item Dimension $6$ is the first case where products appear. Here we have two types of basis elements:
\begin{itemize}
\item Basic $3$-forests with two square vertices and a single round vertex attached to each square vertex. There are $\mybin{6}{3}/2=10$ such basis elements. The group $H_2^6(M_2^{(3)}(6))$ is generated by these basic (but non-elementary) 3-forests, for each of them clearly is the product (up to a sign) of two elementary $3$-forests of dimension 3 (the superposition of its two components).
\item Elementary $3$-forests with one square vertex connected to four round vertices. There are $\mybin{5}{3}=10$ such basis elements (all of them being linearly independent modulo product-decomposable elements), by an analysis similar to the ones in previous items. This corresponds to the summand $H_1^6(M_2^{(3)}(6))$.
\end{itemize}
\item Basic 3-forests in $H^7(M_2^{(3)}(6))$ are necessarily of the form
\begin{center}
\begin{tikzpicture} 
\node (r2) [circle, draw, minimum width=5mm] at (0.75,1) {$i$};
\node (A1) [rectangle, draw, minimum width=1cm] at (0,0) {$A$};
\draw[line width=0.5pt] (A1.north east)--(r2.255);
\begin{scope}[xshift=2.5cm]
\node (r4) [circle, draw, minimum width=5mm] at (0.75,1) {$j$};
\node (A2) [rectangle, draw, minimum width=1cm] at (0,0) {$B$};
\draw[line width=0.5pt] (A2.north east)--(r4.255);
\end{scope}
\draw[line width=0.5pt] (A1.east)--(A2.west);
\end{tikzpicture}
\end{center}
(orientation matters, as well as isolated round vertices are being ignored) and we can assume without loss of generality that $i<j$. Such a basis element is the product of two elementary 3-forests, one of dimension 3 and one of dimension 4, namely
\begin{center}
\begin{tikzpicture} 
\node (r2) [circle, draw, minimum width=5mm] at (0.75,1) {$i$};
\node (A1) [rectangle, draw, minimum width=1cm] at (0,0) {$A$};
\draw[line width=0.5pt] (A1.north east)--(r2.255);
\begin{scope}[xshift=2.75cm]
\node (r3) [circle, draw, minimum width=5mm] at (-0.75,1) {$i$};
\node (r4) [circle, draw, minimum width=5mm] at (0.75,1) {$j$};
\node (A2) [rectangle, draw, minimum width=1cm] at (0,0) {$B$};
\draw[line width=0.5pt] (A2.north west)--(r3.285);
\draw[line width=0.5pt] (A2.north east)--(r4.255);
\end{scope}
\node at ([yshift=2mm]$(A1.east)!0.5!(A2.west)$) {$\cdot$};
\end{tikzpicture}\raisebox{0em}{,}
\end{center}
where isolated round vertices in both factors have been omitted. Note there are $\mybin{6}{3}/2=10$ basis elements in dimension $7$.
\item Finally, observe it is not possible to construct generators of dimension greater than 7, because we can add neither more vertices nor more edges.
\end{itemize}
\end{example}

Elements of maximal dimension in Example~\ref{exampleYuzvinsky} are $k$-forests having a single semilinear tree component of the form
\begin{center}
\begin{tikzpicture} 
\node (r1) [circle, draw, minimum width=5mm] at (-0.75,1) {};
\node (A1) [rectangle, draw, minimum width=1cm] at (0,0) {$A_1$};
\draw[line width=0.5pt] (A1.north west)--(r1.285);
\begin{scope}[xshift=2.5cm]
\node (r3) [circle, draw, minimum width=5mm] at (-0.75,1) {};
\node (A2) [rectangle, draw, minimum width=1cm] at (0,0) {$A_2$};
\draw[line width=0.5pt] (A2.north west)--(r3.285);
\end{scope}
\draw[line width=0.5pt] (A1.east)--(A2.west);
\begin{scope}[xshift=6cm]
\node (r5) [circle, draw, minimum width=5mm] at (-0.75,1) {};
\node (A3) [rectangle, draw, minimum width=1cm] at (0,0) {$A_m$};
\draw[line width=0.5pt] (A3.north west)--(r5.285);
\end{scope}
\draw[line width=0.5pt] (A2.east)--(A3.west)node[midway, fill=white]{$\dots$};
\end{tikzpicture}
\end{center}
More generally:

\begin{lemma}\label{miversion}
Let $n,k,d$ be as in Theorem~\ref{sum}. Elements of maximal dimension in $H^*(\mdkn)$ are given by sums of basic $k$-forests having a single component which is a semilinear tree with $\floor{\frac{n}{k}}$ square vertices and $n-(k-1)\floor{\frac{n}{k}}$ round vertices.
\end{lemma}
\begin{proof}
The map $H^\ast\left(\mdkn;\mathbb{Z}\right)\to H^\ast \left(\mdkn;\mathbb{Z}_2\right)$ induced by mod-2 reduction of coefficients yields an isomorphism after tensoring with $\mathbb{Z}_2$, so that it suffices to prove this lemma for $\mathbb{Z}_2$ coefficients. Consequently, we can ignore all orientation and sign conventions. In addition, it suffices to check the stated characterization for basic $k$-forests of maximal dimension. We start by noticing that such a basic $k$-forest $f$ cannot have isolated round vertices (for any such vertex can be attached to some square vertex of $f$ to produce a basic $k$-forest of larger dimension), and must have a single semilinear tree component (otherwise a basis element of larger dimension can be constructed by adding edges that concatenate the components of $f$). Let $A_1,\ldots,A_m$ denote the square vertices of $f$, and $b_i$ stand for the number of round vertices attached to~$A_i$, so that
\begin{equation}\label{division}
n=m(k-1)+\sum_{i=1}^{\blue{m}}b_i
\end{equation}
as there are no isolated round vertices. We claim that
\begin{equation}\label{sumita}
0\leq\sum_{i=1}^{\blue{m}}(b_i-1)<k.
\end{equation}
The first inequality is obvious as each $b_i$ is positive. If the second inequality fails, then $k$ of the round vertices, except for the greatest round vertex attached to each square vertex, can be detached from its corresponding $A_i$, to yield a smaller-dimensional $k$-tree $f'$. The integers corresponding to the detached round vertices can then be assembled into a new elementary basic $k$-forest that can further be concatenated to $f'$ to yield a basic $k$-forest $f''$. By construction, $\deg(f')=\deg(f)-k\blue{(d-1)}$, while $\deg(f'')=\deg(f')+d(k-\blue{2})+2(d-1)$, which yields $\deg(f'')>\deg(f)$, as \blue{$k\geq 3$}, contradicting the maximality of $f$. This proves~(\ref{sumita}). The conclusion of the lemma now follows from~(\ref{division}) and~(\ref{sumita}): $n=mk+b$, where $b:=\sum_{i=1}^{\blue{m}}(b_i-1)$ is in fact the residue in the division of $n$ by $k$ (so that $m=\floor{\frac{n}{k}}$).
\end{proof}

\begin{corollary}\label{corolarioamiversion}
Let $n,k,d$ be as in Theorem~\ref{sum}. The largest (respectively lowest) positive dimension where the cohomology of $\mdkn$ is non-zero equals $ma+(d-1)(m+b-1)$ (respectively $a$), where $m=\floor{\frac{n}{k}}$, $a=d(k-1)-1$ and $b=n-mk$ (so that $\,0\leq b<k$).
\end{corollary}
\begin{proof}
The first observation in Remark~\ref{decomposition} below yields the assertion about the bottom non-trivial dimension.
Lemma~\ref{miversion} yields the assertion about the top non-trivial dimension.
\end{proof}

\begin{remark}\label{decomposition}
As illustrated in Example~\ref{exampleYuzvinsky}, any subset of $\nn$ with $k$ elements determines (up to a sign) a cohomology class of minimal dimension (i.e. dimension $a$ in the notation of Corollary~\ref{corolarioamiversion}): an elementary $k$-forest with a single attached round vertex (and some prescribed orientations). More generally, choosing $mk$ elements of $\nn$, and partitioning these elements into $m$ subsets of cardinality $k$, say $P_1 \sqcup P_2 \sqcup \cdots \sqcup P_m$, we can form a basic $k$-forest of dimension $ma$ which, in addition, factors (up to a sign) as a product of $m$ elementary minimal-dimension $k$-forests, namely those determined by each $P_i$. This observation will be the basis to construct, in the next section, a number of relevant cohomology classes in cartesian products of $\mdkn$.
\end{remark}

Throughout the rest of the paper we make free use of the description of the cohomology ring $H^*(\mdkn)$ (either with coefficients in $\mathbb{Z}$ or $\mathbb{Z}_2$) in terms of $k$-forests, and its properties. Also, as explained in the introduction, the conditions $d\geq2$ and $n>k\geq3$ will be in force.

\section{Cup-length and zero-divisors cup-lenght}

In this section we combine information coming from the connectivity, the homotopy dimension, and the cohomology ring of $\mdkn$ in order to estimate the Lusternik-Schnirelmann category (cat) and topological complexity (TC) of \mdkn. Cohomology coefficientes will be taken in $R$, where either $R=\mathbb{Z}$ or $R=\mathbb{Z}_2$. Assertions made without specifying the ring $R$ are meant to hold for both options of $R$.

\begin{definition}
Given a path-connected space $X$: 
\begin{itemize}
\item The \emph{cup-length of $X$,} $\text{cl}(X)$, is the maximal integer $\ell$ such that there exist cohomology classes $u_1,\dots, u_{\ell} \in \widetilde{H}^\ast(X)$ with non-trivial product $u_1 \cdots  u_{\ell}$.
\item The \emph{zero-divisor cup-length of $X$,} $\text{zcl}(X)$, is the maximal integer $\ell$ such that there exist cohomology classes $z_1,\ldots,z_\ell \in H^\ast (X\times X)$, each with trivial restriction under the diagonal inclusion $\Delta:X \hookrightarrow X\times X$, and so that the product $z_1\cdots z_\ell$ is non-zero. Each such cohomology class $z_i$ is called a zero-divisor for $X$.
\item More generally, for $s\geq2$, the $s$-th zero-divisor cup-length of $X$, $\text{zcl}_s(X)$, is the maximal integer $\ell$ such that there exist cohomology classes $z_1,\ldots,z_\ell \in H^\ast (X^s)$ with trivial restriction under the diagonal inclusion $\Delta: X \hookrightarrow X^s$, and so that the product $z_1\cdots z_\ell$ is non-zero. Each of these classes $z_i$ is called a $s$-th zero-divisor for $X$.
\end{itemize}
\end{definition}

Although both cl$(X)$ and zcl$_s(X)$ can be defined in a more general setting (using cohomology with local coefficients), here we only need to make use of $R$ (untwisted) coefficients. As in Theorem~\ref{omnibus}, it will be convenient to set $\text{zcl}_1=\text{cl}$ and $\text{zcl}_2=\text{zcl}$. If the ring under consideration needs to be specified we will use the more explicit notation $\text{zcl}_s
^R$ (or $\text{zcl}^R$, $\text{cl}^R$).

\begin{lemma}\label{lengths} For $s\geq1$, $\mathrm{zcl}_s(\mdkn)=s\floor{\frac{n}{k}}$.\end{lemma}

\begin{proof}
Recall we assume $k<n$, in particular $m:=\floor{\frac{n}{k}}\geq1$. We start working with $\mathbb{Z}_2$-coefficients. For $k\leq i\leq n$, let $x_i$ be the elementary $k$-forest \begin{equation*}\tikz[baseline={([yshift=3pt]current bounding box.south)}]{\node (A) [rectangle, draw, minimum width=8mm]{$\scriptstyle A$};\node (B) [circle, draw, inner sep=0.5pt, minimum width=4.5mm] at ([shift={(3.5mm,3.5mm)}]A.north east){$\scriptstyle B$};\draw[line width=0.5pt] (A.north east)--(B.235);}\end{equation*} where $A=\{i-k+1,\, i-k+2,\, \dots, i-1\}$, $B=i$ and the remaining indices of $\nn$ lie on isolated round vertices. Similarly, let $\widetilde{x}_i$ be the elementary $k$-forest as above where now $A=\{1,\, i-k+2,\, \dots,\, i-1 \}$, $B=i$. Note that $x_k=\widetilde{x}_k$, however $x_i$ and $\widetilde{x}_i$ are different basis elements for $i>k$. Furthermore, by Remark~\ref{decomposition}, the products
\begin{gather}
x_k \cdot x_{2k} \cdot \cdots \cdot x_{(m-1)k} \cdot x_{mk} \label{xis}\\
x_{k+1} \cdot x_{2k+1} \cdot \cdots \cdot x_{(m-1)k+1} \cdot \widetilde{x}_{mk} \label{xtildeis}
\end{gather}
are basic $k$-forests and thus non-zero. This yields in particular the inequality 
\begin{equation}\label{desigualdadparacategoria}
\floor{\frac{n}{k}}\leq \text{cl}^{\mathbb{Z}_2}(\mdkn).
\end{equation}
Note that the elements in~(\ref{xis}) and~(\ref{xtildeis}) are different basic $k$-forests except for $m=1$, in which case both coincide with $\widetilde{x}_k = x_k$. 

Since the mod-$2$ reduction map $\mathbb{Z}\to \mathbb{Z}_2$ induces a ring epimorphism, we have $\text{cl}^{\mathbb{Z}_2}(\mdkn) \leq \text{cl}^{\mathbb{Z}}(\mdkn)$. Therefore, it suffices to show $$\text{cl}^{\mathbb{Z}}(\mdkn)\leq \floor{\frac{n}{k}}.$$ Thus, we now switch to $\mathbb{Z}$-coefficients, noticing that it suffices to show that the product of any set of $m+1$ elementary oriented $k$-forests vanishes. In turn it suffices to show that there are no oriented $k$-forests with $m+1$ square vertices. But any such $k$-forest would have, in addition to the integers inside the $m+1$ square vertices, at least one integer attached to each square vertex, making a total of at least $(m+1)k$ integers inside $\nn$. This is impossible for $m=\floor{\frac{n}{k}}$.

Next we bound from below the zero-divisors cup-length zcl$_2(\mdkn)$. Working again with $\mathbb{Z}_2$-coefficients, we can consider the zero-divisors in $H^\ast(\mdkn)\otimes H^\ast(\mdkn)$ given as
\begin{align*}
y_{i,1} &= 1 \otimes x_{ik+1} + x_{ik+1}\otimes 1, \enskip \text{ for } 1\leq i<m, \\
y_{m,1} &= \begin{cases} 1\otimes x_{k+1}+x_{k+1}\otimes 1, \enskip\text{if }m=1 \mbox{ (recall $k<n$)}; \\
1\otimes \widetilde{x}_{mk} +\widetilde{x}_{mk}\otimes 1, \enskip \text{if } m>1,
\end{cases}\\
y_{i,2} &= 1 \otimes x_{ik} + x_{ik} \otimes 1, \enskip\text{ for }1\leq i\leq m.
\end{align*}
If $m=1$, $y_{1,1}y_{1,2} = x_{k+1}\otimes x_{k}+x_{k}\otimes x_{k+1}\neq0,$ showing $2m \leq \text{zcl}^{\mathbb{Z}_2}(\mdkn)$. For $m>1$ observe that the square vertex in $x_{ik}$ intersects the square vertex in $x_{ik+1}$ (as $k\geq3$) and so their product is zero. Consequently $y_{i,1}y_{i,2} = x_{ik} \otimes x_{ik+1} + x_{ik+1}\otimes x_{ik}$ for $i<m$. Likewise, $y_{m,1} y_{m,2} = \widetilde{x}_{mk}\otimes x_{mk} + x_{mk} \otimes \widetilde{x}_{mk}$. Note also that each product $x_{ik+1}x_{(i+1)k}$ vanishes (cf.~Example~\ref{ejemplito1}), as well as the product $x_k \widetilde{x}_{mk}$, so we have
\[
\prod_{i=1}^m y_{i,1}y_{i,2} = \left(\left(\prod_{i=1}^{m-1} x_{ik+1}\right)  \widetilde{x}_{mk}\right) \otimes \prod_{i=1}^m x_{ik}+\prod_{i=1}^m x_{ik}\otimes \left(\left(\prod_{i=1}^{m-1} x_{ik+1}\right)  \widetilde{x}_{mk}\right),
\]
which is the (symmetric) sum of the tensor product of the basis elements \eqref{xis} and \eqref{xtildeis}. This gives again $2m \leq \text{zcl}^{\mathbb{Z}_2}(\mdkn)$. Furthermore, the surjectivity argument used in the case of cup-length allows us to assemble $\mathbb{Z}$-zero-divisors (of the form $1\otimes z -z\otimes 1$, rather than $1\otimes z + z \otimes 1$) giving $2m\leq \text{zcl}^{\mathbb{Z}}\left(\mdkn\right)$.

The fact that the latter inequality is sharp (with either $\mathbb{Z}$ or $\mathbb{Z}_2$ coefficients) will follow once we observe that, actually, the product of any $2m+1$ positive-dimensional basic tensors $b_i=u_i \otimes v_i$ in $H^\ast(\mdkn) \otimes H^\ast(\mdkn)$ vanishes. In the product
\[
\prod_{i=1}^{2m+1} b_i = (u_1  \cdots  u_{2m+1}) \otimes (v_1  \cdots  v_{2m+1})
\]
one of the factors $u_1  \cdots  u_{2m+1}$ or $v_1  \cdots  v_{2m+1}$ vanishes as it is the product of at least $m+1$ positive-dimensional cohomology classes.

These arguments generalize easily to yield zcl$_s^R(\mdkn)=sm$. Working with $R=\mathbb{Z}_2$, consider the $s$-th zero divisors in $H^\ast(\mdkn)^{\otimes s}$:
\begin{align*}
z_{i,1} &= 1\otimes x_{ik+1} \otimes 1 \otimes \cdots \otimes 1 + x_{ik+1}\otimes 1 \otimes \cdots \otimes 1, \enskip \text{for }1\leq i<m, \\
z_{m,1} &= \begin{cases} 1 \otimes {x}_{k+1} \otimes 1\otimes \cdots \otimes 1 + {x}_{k+1} \otimes 1\otimes\cdots \otimes 1,\enskip \text{for } m=1; \\
1 \otimes \widetilde{x}_{mk} \otimes 1\otimes \cdots \otimes 1 + \widetilde{x}_{mk} \otimes 1\otimes\cdots \otimes 1,\enskip\text{for }m>1, 
\end{cases}\\
z_{i,j} &= 1 \otimes \cdots \otimes 1 \otimes \underbrace{ x_{ik}}_{j-\text{th}} \otimes 1 \otimes \cdots \otimes 1 + x_{ik} \otimes 1 \otimes \cdots \otimes 1, \enskip\text{for } 1\leq i\leq m \text{ and } 2\leq j \leq s.
\end{align*}
Direct calculation yields $\prod_{i=1}^m \prod_{j=1}^s z_{i,j} \neq 0$. For instance, if $m>1$, we have
\begin{align*}
\prod_{j=1}^s z_{i,j} &= x_{ik+1} \otimes x_{ik} \otimes x_{ik} \otimes \cdots \otimes x_{ik} + x_{ik} \otimes x_{ik+1} \otimes x_{ik} \otimes \cdots \otimes x_{ik}\\
&= (y_{i,1} y_{i,2}) \otimes x_{ik} \otimes \cdots \otimes x_{ik},
\intertext{for $i<m$, and }
\prod_{j=1}^s z_{m,j} &= \widetilde{x}_{mk} \otimes x_{mk} \otimes \cdots \otimes x_{mk}+x_{mk} \otimes\widetilde{x}_{mk} \otimes x_{mk}\otimes  \cdots \otimes x_{mk}\\
&= (y_{m,1} y_{m,2})\otimes x_{mk} \otimes \cdots \otimes x_{mk}.
\end{align*}
So
\[
\prod_{i=1}^m \prod_{j=1}^s z_{i,j} = \left(\prod_{i=1}^{m} y_{i,1} y_{i,2} \right) \otimes \prod_{i=1}^m x_{ik} \otimes \cdots \otimes \prod_{i=1}^m x_{ik} \neq 0.
\]
Therefore $sm\leq\text{zcl}_s^R(\mdkn)$ for $R=\mathbb{Z}_2$ and, as above, for $\mathbb{R}=\mathbb{Z}$. The latter inequality is sharp by considerations similar to those in the case $s=2$.
\end{proof}

Next we make use of these bounds to estimate the category and all topological complexities of \mdkn.

\begin{theorem}\label{cat} For $s\geq 1$, \emph{TC}$_s\left(\mdkn\right)$  is bounded by
\[
s\floor{\dfrac{n}{k}} \leq \mathrm{TC}_s\left(\mdkn\right) \leq s\left(\floor{\dfrac{n}{k}}+\floor{\dfrac{\left(\floor{\frac{n}{k}}+b-1\right)(d-1)}{a}}\right),
\]
where $a=d(k-1)-1$ and $n=k\floor{\frac{n}{k}}+b$ with $0 \leq b < k$.
\end{theorem}
\begin{proof} The lower bound follows from Lemma~\ref{lengths} and the standard fact that $\text{zcl}_s\leq\text{TC}_s$. For $s=1$, the upper bound follows from the well known bound for the Lusternik-Schnirelmann category of a space in terms of its connectivity and homotopy dimension, namely
\[
\text{cat}(\mdkn) \leq \dfrac{\text{hdim}(\mdkn)}{\text{conn}(\mdkn)+1},
\]
where hdim$(\mdkn)$ (respectively $\text{conn}(\mdkn)$) is the cellular homotopy dimension (respectively, connectivity) of $\mdkn$. Indeed, under our general hypothesis ($d\geq 2$, $k \geq 3$), $\mdkn$ is simply connected \cite[Theorem 1.2]{Kallel2016} and has torsion-free $\mathbb{Z}$-homology \cite[Proposition 3.9]{dobri2015}, so $$\text{conn}(\mdkn)+1=a,$$ in view of Corollary~\ref{corolarioamiversion}, the Hurewicz Theorem and the universal coefficients theorem, whereas 
\begin{equation}\label{hdim}
\text{hdim}(\mdkn)=ma+(d-1)(m+b-1),
\end{equation}
in view of Corollary~\ref{corolarioamiversion} (so $m=\floor{\frac{n}{k}}$),~\cite[Proposition~4C.1]{MR1867354}, and the universal coefficient theorem. Lastly, for a general $s\geq 2$, the upper bound follows from the well known bound TC$_s\leq s\cdot $cat.
\end{proof}

\begin{corollary}\label{TCshypothesis} For $s\geq 1$, \emph{TC}$_s\left(\mdkn\right)=s\floor{\frac{n}{k}}$ provided $n-(k-1)\floor{\frac{n}{k}}<\frac{dk-2}{d-1}$. 
\begin{proof}
With the notation of Theorem~\ref{cat}, TC$_s\left(\mdkn\right)=sm$ if

\centerline{{\large\mbox{$\frac{(m+b-1)(d-1)}{a}$}}$<1$} 

or, equivalently, $m+b<\frac{dk-2}{d-1}$.
\end{proof}
\end{corollary}

\section{Obstruction theory}

Corollary~\ref{TCshypothesis} yields the case $n-(k-1)\floor{\frac{n}{k}}<\frac{dk-2}{d-1}$ in Theorem~\ref{omnibus}. In this section we address the remaining instances in Theorem~\ref{omnibus}, i.e., when the latter inequality is an equality. In fact, we improve by $s$ units the upper bound in Theorem~\ref{cat} for all cases where $(\floor{\frac{n}{k}}+b-1)(d-1)$ is divisible by $a$.

The following fact is standard, see for instance \cite[Theorem 3.1]{Gonzalez2015}.

\begin{theorem}\label{obstruccion} Let $p:E \to B$ be a fibration with fiber $F$ whose base $B$ is a CW complex. Assume $p$ admits a section $\phi$ over the $s$-skeleton $B^{(s)}$ of $B$ for some $s\geq 1$. If $F$ is $s$-simple and the obstruction cocycle to the extension of $\phi$ to $B^{(s+1)}$ lies in the cohomology class
\[
\eta\in H^{s+1}\big(B;\left\{\pi_s(F)\right\}\big),
\]
then $p(\ell)$ (the $(\ell+1)$-th fiberwise join power of $p$) admits a section over $B^{(s+1)(\ell+1)-1}$ whose obstruction cocycle to extending to $B^{(s+1)(\ell+1)}$ belongs to the cohomology class
\[
\eta^{\ell+1} \in H^{(s+1)(\ell+1)} \big(B;\{\pi_{s\ell+s+\ell}(F^{\ast (\ell+1)})\}\big).
\]
\end{theorem}

In Theorem~\ref{obstruccion}, $\eta^{\ell+1}$ stands for the image of the $(\ell+1)$-fold cup-power of $\eta$ under the $\pi_1(B)$-homomorphism of coefficients
\begin{equation}\label{mapofcoef}
\pi_s(F)^{\otimes (\ell+1)}\to \pi_{s\ell+s+\ell}(F^{\ast (\ell+1)})
\end{equation}
given by iterated join of homotopy classes. We use Theorem~\ref{obstruccion} when $B$ is simply connected, so that all cohomology groups above have trivial systems of coefficients, and when $F$ is $(s-1)$-connected, so that~(\ref{mapofcoef}) is an isomorphism, and $\eta^{\ell+1}$ is really the $(\ell+1)$-st cup-power of $\eta$. In addition, our connectivity hypothesis on $F$ implies that $\eta$ and $\eta^{\ell+1}$ are the primary obstructions for sectioning $p$ and $p(\ell)$, respectively, and thus they are well defined (no indeterminacy). Lastly, since the pull-back $p^\ast(p)$ admits a tautological section, we have $p^\ast(\eta)=0$ \emph{a fortiori}.

As in the previous section, we denote by $m$ and $b$ the quotient and remainder, respectively, of the division of $n$ by $k$. The role of $p$ in Theorem~\ref{obstruccion} will be played by the based path-space fibration 
\begin{equation}\label{bpfib}
\Omega \mdkn \to P_0(\mdkn) \overset{e_1}{\to} \mdkn.
\end{equation}
We analyze the obstructions for having cat$\left(\mdkn\right)=\text{secat}(e_1)\leq m+i-1$, where $i$ is a positive integer, or, equivalently, for having secat$(e_1(m+i-1))=0$, where as in Theorem~\ref{obstruccion}
\[
\bigast_{m+i} \left(\Omega \mdkn\right) \to J_{m+i-1} \left(P_0(\mdkn)\right) \overset{e_1(m+i-1)}{\longrightarrow} \mdkn
\]
stands for the $(m+i)$-fold fiberwise join-power of $e_1$ (so $\ell=m+i-1$ in Theorem~\ref{obstruccion}). Since $\Omega \mdkn$ is  $(a-2)$-connected, there are no obstructions for picking a section $\phi$ over the $(a-1)$-skeleton of $\mdkn$ (so $s=a-1$ in Theorem~\ref{obstruccion}). Therefore, if $\eta\in H^a\left(\mdkn;\pi_{a-1}\left(\Omega \mdkn\right)\right)$ stands for the primary obstruction for sectioning $e_1$, then the primary obstruction for sectioning $e_1(m+i-1)$ is the $(m+i)$-st cup-power
\begin{equation*}
\begin{split}
\eta^{m+i} \in H^{a(m+i)}\left(\mdkn ; \pi_{a(m+i)-1} \left(\bigast_{m+i} \left(\Omega \mdkn\right)\right)\right)\\=H^{a(m+i)}\left(\mdkn;\left(\pi_{a-1}\left(\Omega\mdkn\right)\right)^{\otimes(m+i)}\right).
\end{split}
\end{equation*}
In view of \eqref{hdim}, all potential obstructions for sectioning $e_1(m+i-1)$ lie in trivial groups when $a(m+i)>ma+(d-1)(m+b-1)$. For $i=1$, this of course yields a direct obstruction-theoretic argument for the inequality $\cat(\mdkn)\leq m$ in Corollary~\ref{TCshypothesis}. Yet, we need the cup-length arguments in the previous section in order to deal with the case where the primary obstruction $\eta^{m+i}$ does not lie in a trivial group. Actually, we next prove the triviality of the $(m+1)$-st cup-power of any element in $H^a\left(\mdkn;\pi_{a-1}\left(\Omega\mdkn \right)\right)$. 

\begin{lemma}\label{primaryobstructionpowervanishes}
Recall $m=\floor{\frac{n}{k}}$ and $a=dk-d-1$. Any element $$\eta \in H^a \left(\mdkn;\pi_{a-1}\left(\Omega \mdkn\right)\right)$$ has trivial $(m+1)$-st cup-power.
\begin{proof}
The Hurewicz theorem and the considerations in Section~\ref{secciondescripciondecohomologia} (see particularly Theorem~\ref{sum} and Remark~\ref{decomposition}) show that the coefficient group $\pi_{a-1}\left(\Omega\mdkn\right)$ is free abelian of rank $\binom{n}{k}$. So, in terms of the decomposition $H^a \left(\mdkn;\bigoplus_{\binom{n}{k}} \mathbb{Z}\right)=\bigoplus_{\binom{n}{k}} H^a\left(\mdkn;\mathbb{Z}\right)$, we write $\eta = \sum_{\binom{n}{k}} \eta_j$. The naturality of cup-product on coefficients yields $\eta^{m+1} = \left(\sum \eta_j\right)^{m+1} = \sum \eta_{j_1} \cdots \eta_{j_{m+1}}$ where each summand $\eta_{j_1}  \cdots  \eta_{j_{m+1}}$ stands for the image of the cup-product $\eta_{j_1}\cup \cdots \cup \eta_{j_{m+1}}\in H^{a(m+1)}\left(\mdkn;\mathbb{Z}\right)$ under the map induced on coefficients by
$$\xymatrix{\mathbb{Z}=\mathbb{Z}\otimes\cdots\otimes\mathbb{Z}\ar[rr]^{\iota_{j_1}\otimes \cdots \otimes \iota_{j_{m+1}}} && \left(\bigoplus_{\binom{n}{k}}\mathbb{Z}\right)_{\text{\normalsize $\;.$}}^{\otimes(m+1)}}$$
Here $\iota_r:\mathbb{Z}\hookrightarrow \bigoplus_{\binom{n}{k}} \mathbb{Z}$ stands for the inclusion into the $r$-th summand. The triviality of $\eta^{m+1}$ then follows from that of each $\eta_{j_1} \cup \cdots \cup \eta_{j_{m+1}}$ which, in turn, follows from the case $s=1$ in Lemma~\ref{lengths}.
\end{proof}
\end{lemma}

Returning to the discussion prior to Lemma~\ref{primaryobstructionpowervanishes}, we next prove a strengthening of Theorem~\ref{cat}, from which Theorem~\ref{omnibus} follows as an immediate consequence.

\begin{theorem}\label{catimproved} For $s\geq 1$, \emph{TC}$_s\left(\mdkn\right)$  is bounded by
\[
s\floor{\dfrac{n}{k}} \leq \mathrm{TC}_s(\mdkn) \leq s\left(\floor{\dfrac{n}{k}}+\mfloor{\dfrac{\left(\floor{\frac{n}{k}}+b-1\right)(d-1)}{a}}\right),
\]
where $a=d(k-1)-1$ and $n=k\floor{\frac{n}{k}}+b$ with $0 \leq b < k$.
\end{theorem}
\begin{proof} 
We only need to focus on the cases not covered by Theorem~\ref{cat}, i.e., those satisfying
\begin{equation}\label{hip0tesis}
ai=(d-1)(m+b-1)
\end{equation}
for some positive integer $i$. Further, in such a case, the well known estimate $\TC_s\leq s\cdot\cat$ implies that it suffices to prove
\begin{equation}\label{baste}
\cat(\mdkn)\leq m+i-1.
\end{equation}
But Lemma~\ref{primaryobstructionpowervanishes} (and the discussion preceding it) give the vanishing of the primary obstruction for~(\ref{baste}), i.e., for sectioning $e_1(m+i-1)$, whereas the rest of the higher obstructions lie in trivial groups ---by~(\ref{hdim}) and~(\ref{hip0tesis}). 
\end{proof}

For $k$ fixed the function $f_k(d)=\dfrac{dk-2}{d-1}$ is decreasing, so that Theorem~\ref{omnibus} applies for more values of $n$ when $d=2$. The following assertion identifies the first complete interval of values of $n$ where Theorem~\ref{omnibus} holds for $d=2$. Tables~\ref{d2}, \ref{d5} and \ref{d10} illustrate the broader scope of Theorem~\ref{omnibus}.

\begin{corollary} If $k\neq n\leq k^2+k-2$ and $s\geq 1$, then $\TC_s(M_2^{(k)}(n))=s\floor{\frac{n}{k}}$.
\end{corollary}

Tables \ref{d2}, \ref{d5}, and \ref{d10} show the values of cat$(\mdkn)$ in cases determined by Theorem~\ref{omnibus}, for $d=2$, $d=5$ and $d=10$, respectively. Values of TC and TC$_s$ can then be read off by multiplying by $2$ or $s$, respectively. For example TC$(M_2^{(8)}(40))=10$ and TC$_s(M_2^{(8)}(40))=5s$. Shading tones in these tables indicate cases with a common value of $\floor{\frac{n}{k}}$, while the actual tabulated numbers indicate the values of $\mathrm{cat}(\mdkn)$ coming from Theorem~\ref{omnibus}. Instances where the equality cat$(\mdkn)=\floor{\frac{n}{k}}$ is not established by Theorem~\ref{omnibus} are indicated with a question mark. Note from these tables that the value of $\text{cat}(\mdkn)$ is determined when $n$ is "close" to $k$ (top right region). However $\text{cat}(\mdkn)$ becomes indetermined as $n$ "moves away" from $k$ (bottom left region). The general structure of the tables is relatively simple: column $k$ is divided into blocks of size $k$ (except for the very first block, whose size is $k-1$) sharing a common value of $\floor{\frac{n}{k}}$. For the top blocks, the common value is the answer for cat$\big(\mdkn\big)$, except for lower blocks, which start having instances where the condition $\floor{\frac{n}{k}}+b\leq \frac{dk-2}{d-1}$ in Theorem~\ref{omnibus} fails.

In principle, the obstruction techniques used in this section for the base path evaluation map~(\ref{bpfib}) could be used directly with the fibrations defining the higher topological complexities $\TC_s$. It is interesting to remark that such a strategy does not seem to lead to any improved $\TC_s$ upper bounds for the manifolds $\mdkn$; instead, it suggests the possibility that the gap in Theorem~\ref{cat} would have to be resolved by improving the lower bound. For such a task, non-trivial Massey products holding in non-formal spaces $\mdkn$ might be a way to formalize the suggested phenomenon.


\begin{table}
\centering
\caption{Lusternik-Schnirelmann category values for $M_2^{(k)}(n)$}
\label{d2}
\includegraphics[width=\linewidth]{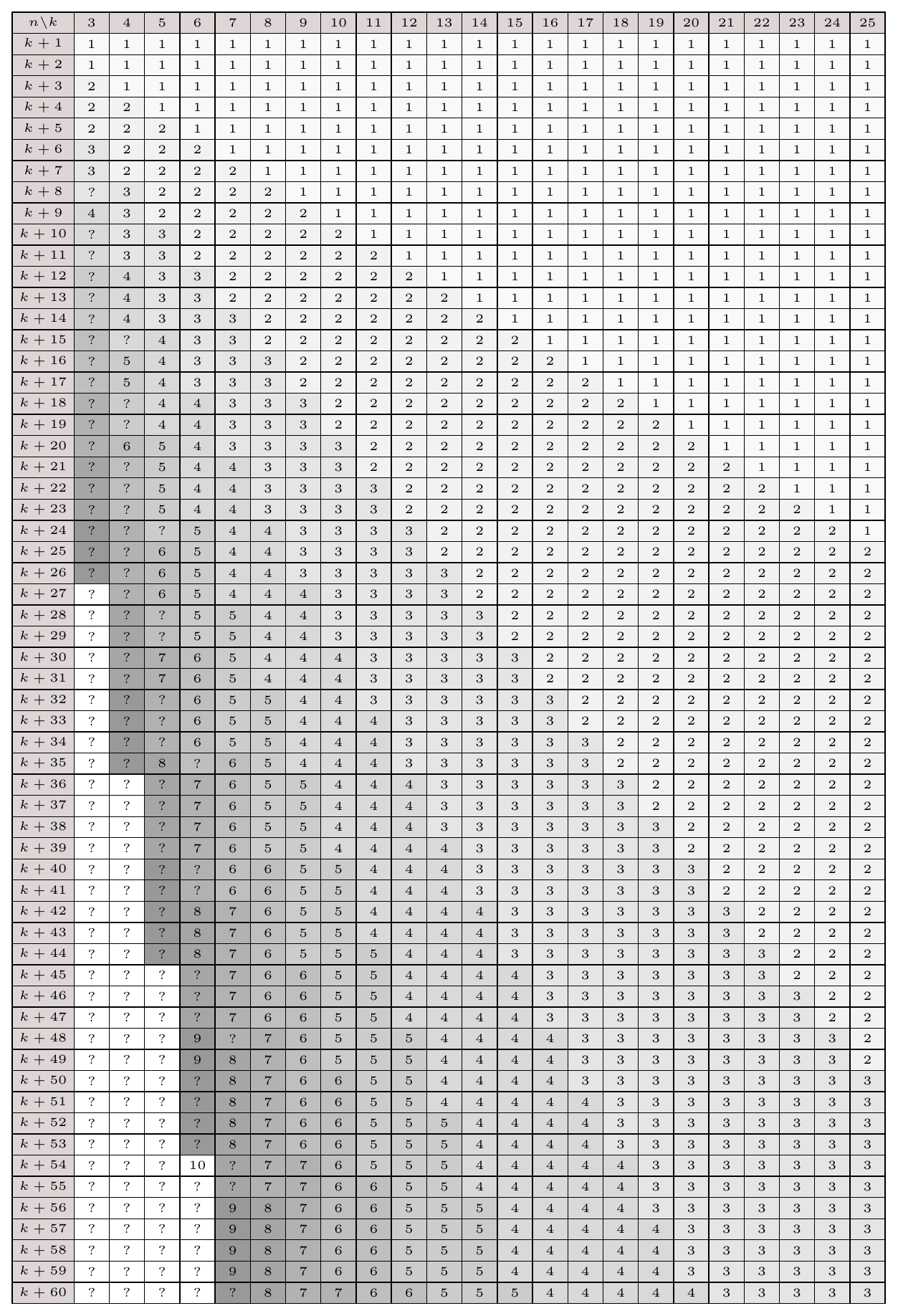}
\end{table}

\begin{table}
\centering
\caption{Lusternik-Schnirelmann category values for $M_5^{(k)}(n)$}
\label{d5}
\includegraphics[width=\linewidth]{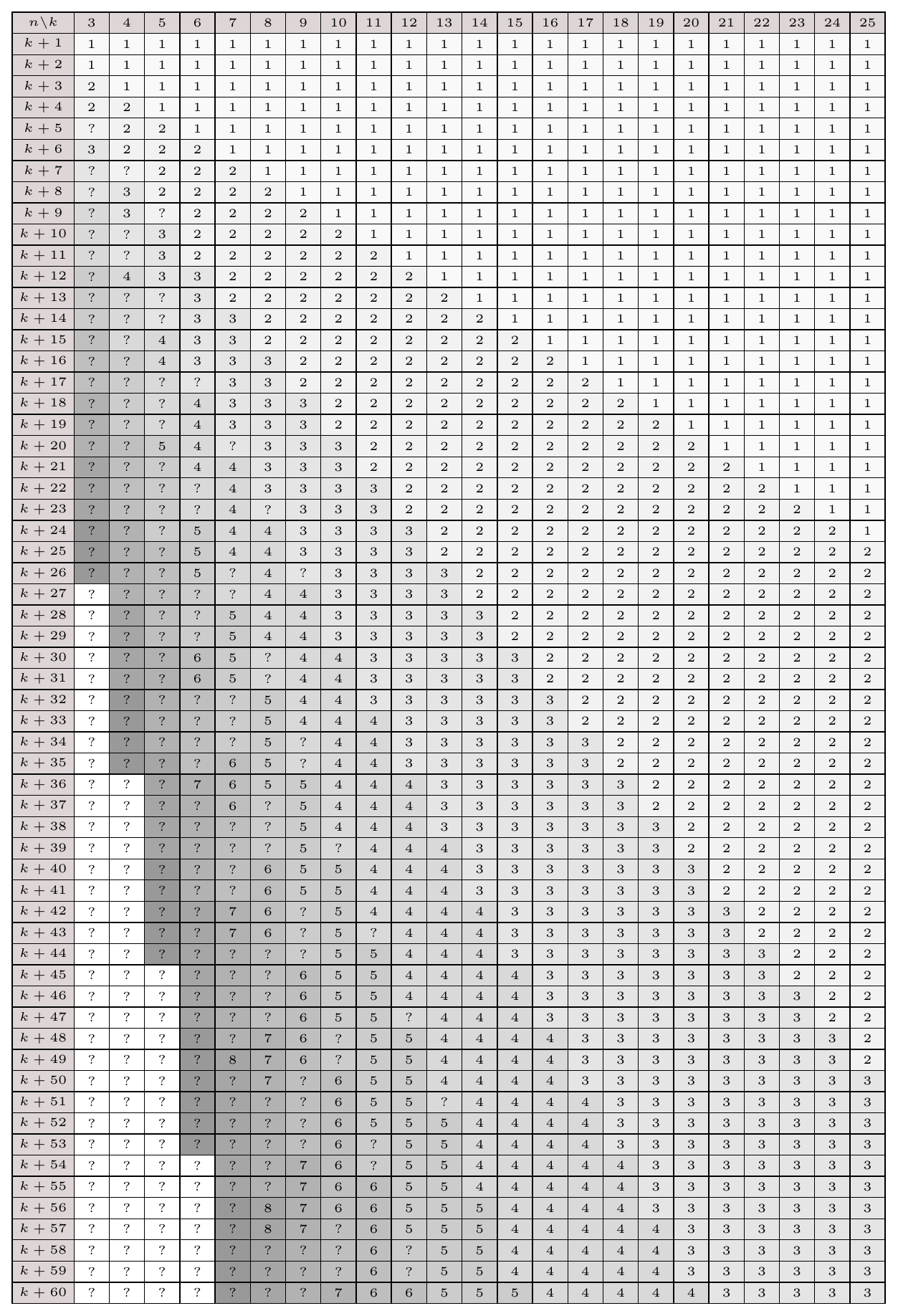}
\end{table}

\begin{table}
\centering
\caption{Lusternik-Schnirelmann category values for $M_{10}^{(k)}(n)$}
\label{d10}
\includegraphics[width=\linewidth]{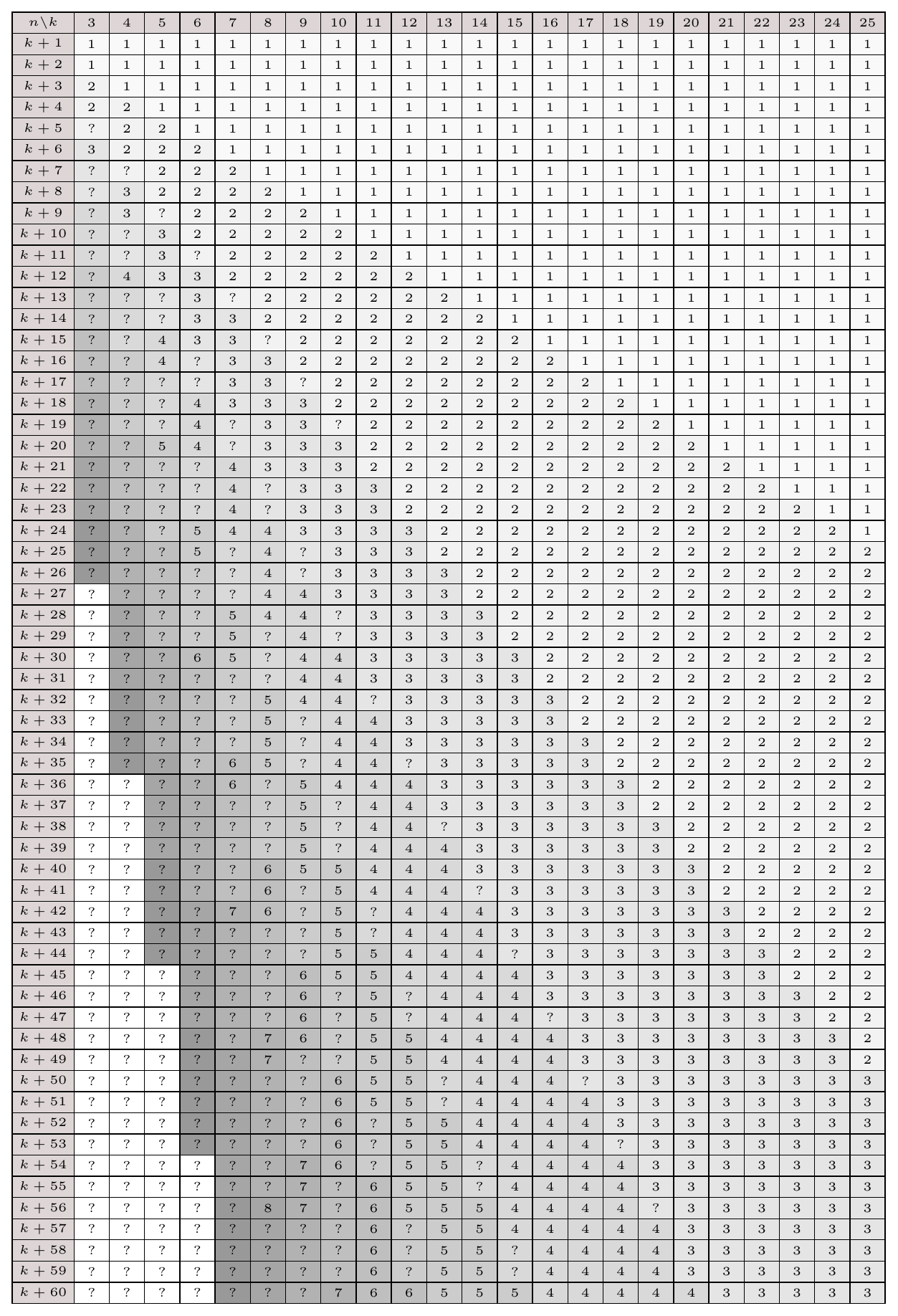}
\end{table}

\medskip
{\small \sc Departamento de Matem\'aticas

Centro de Investigaci\'on y de Estudios Avanzados del I.P.N.

Av.~Instituto Polit\'ecnico Nacional n\'umero 2508

San Pedro Zacatenco, M\'exico City 07000, M\'exico

{\tt jesus@math.cinvestav.mx}

{\tt  jose.leon@cinvestav.mx}}

\end{document}